\documentclass{article}
\usepackage[utf8]{inputenc}
\usepackage[english]{babel}
\usepackage{amsmath}
\usepackage{amssymb}
\usepackage{amsthm}
\usepackage{enumerate}
\usepackage{color,graphicx}
\usepackage{lastpage209}

\newtheorem{theorem}{Theorem} 
\newtheorem{lemma}{Lemma}
\newtheorem{corollary}{Corollary}
\newtheorem{proposition}{Proposition}
\theoremstyle{definition}
\newtheorem{definition}{Definition}
\newtheorem{remark}{Remark}

\newtheorem{example}{Example}

\def\NN{{\mathbb N}}

\def\P{\mathcal P}

\def\mod{{\rm mod}\,}
\renewcommand{\Im}{\mathop{\mathrm{Im}}\nolimits}

\begin{document}

\title{Restrictions on local embeddability into finite semigroups}

\author{Dmitry Kudryavtsev\footnote{Department of Mathematics, University of Manchester, Oxford Rd, Manchester, M13 9PL, United Kingdom. Email dmitry.kudryavtsev@postgrad.manchester.ac.uk}}

\maketitle

\begin{abstract}

In this paper the concept of local embeddability into finite structures (being LEF) is for the class of semigroups is expanded with investigations of non-LEF structures, a closely related generalising property of local wrapping of finite structures (being LWF) and inverse semigroups. The established results include a description of a family of non-LEF semigroups unifying the bicyclic monoid and Baumslag--Solitar groups and establishing that inverse LWF semigroups with finite number of idempotents are LEF.



\end{abstract}

\maketitle

\section*{Acknowledgments}

The work was financially supported by a University of Manchester President's Doctoral Scholar Award.

The author is grateful to his research supervisors, Prof Mark Kambites and Dr Marianne Johnson, for the provided commentary and insightful conversations regarding the subject of the paper.

\section*{Introduction}

The idea of approximating infinite structures by finite ones has been approached in many ways  through different properties such as residual finiteness (an algebraic approach based on homomorphisms), pseudofiniteness (a model theory approach) and soficity (a more topological approach).

In particular, the notion of local embeddability into the class of finite (LEF, for short) structures, which in a sense incorporates all of the aforementioned paths, was initially introduced in \cite{GV} for groups, expanded to general structures in \cite{Bel} and recently examined specifically in the class of semigroups in \cite{K23}.

The main purpose of this work is to continue this examination. We produce more examples of non-LEF behaviour, investigate a closely related property which we name local wrapping of the class of finite (LWF, for short) semigroups, and study connections between general LEF semigroups and semigroups locally embeddable into the class of finite inverse semigroups (iLEF, for short).

Our paper is organised as follows. In Section 1 we provide the definition of LEF and LWF structures, demonstrate that the former implies the latter, recall connection between LEF groups and seigroups and also several examples. Section 2 covers numerous specific restrictions of being LEF. In Section 3 we prove general properties of LWF structures and establish that the key examples of non-LEF semigroups are also non-LWF. Finally, the results of Section 4 specify the theory of LEF and LWF objects to the class of inverse semigroups, and in particular demonstrate that inverse LWF semigroups with finite number of idempotents are also LEF.

\section{Definitions and initial results}

We begin with a definition of the main structures studied in this paper.

\begin{definition}\label{def_lef}
A (semi)group $S$ is called {\em locally embeddable into the class of finite (semi)groups} (an {\em LEF} (semi)group  for short) if for every finite subset $H$ of $S$ there exists a finite (semi)group $F_H$ and an injective function $f_H: H \rightarrow F_H$ such that for all  $x,y \in H$ with $xy \in H$ we have $(xy)f_H = (xf_H)(yf_H)$.
\end{definition}

The following definition encapsulates reason why a given non-LEF semigroup may fail the requirements of Definition \ref{def_lef}.

\begin{definition}\label{def_nemb}
We say that a finite set $H$ with partially defined multiplication which satisfies associativity for any valid inputs is {\em non-embeddable} if there does not exist a finite (semi)group $F_H$ and an injective function $f_H: H \rightarrow F_H$ such that for all  $x,y \in H$ with $xy \in H$ we have $(xy)f_H = (xf_H)(yf_H)$.
\end{definition}

We also recall the following result connecting LEF groups and semigroups.

\begin{proposition}\label{prop_lefgr}\cite[Proposition 1.2]{K23}
A group is an LEF semigroup if and only if it is an LEF group.
\end{proposition}

It has been shown previously that the free groups and semigroups are LEF, while the bicyclic monoid $B = \langle a,b \mid ab=1\rangle$ is not LEF and the set $\{1,a,b,ba\}$ with the partial multiplication inherited from $B$ is non-emdeddable (see \cite[Examples 1.5 and 1.7]{K23} for the semigroup and monoid results).

We introduce an additional notion closely related to the LEF property, inspired by \cite{GV}.

\begin{definition}\label{def_lwf}
We say that a (semi)group $S$ is {\em locally wrapped by the class of finite (semi)groups} (an {\em LWF} (semi)group for short) if for every finite subset $H$ of $S$ there exist a finite (semi)group $D_H$ and a function $d_H: D_H \rightarrow S$ such that $H \subseteq D_H d_H$ and for all $x', y' \in D_H$ with $x'd_H,y' d_H \in H$ we have $(x'y')d_H = (x'd_H)(y'd_H)$.
\end{definition}

\begin{proposition}\label{prop_lefislwf}
Let $S$ be an LEF (semi)group. Then it is LWF.
\end{proposition}

\begin{proof}
If $S$ is finite, the statement is evident. Assume that $S$ is infinite.

Consider a finite subset $H$ of $S$. Let $K=H \cup H^2, s \in S \setminus K$  and consider the (semi)group $F_{K}$ and the function $f_{K}$ which satisfy the requirements of Definition \ref{def_lef} for $K$. We claim that the (semi)group $D_H = F_{K}$ and the function $d_H$ defined by 
\begin{align*}x' d_H & = \begin{cases} x, & \text {if } x \in K \text { and } x f_{K} = x', \\ s, & \text{otherwise,} 
\end{cases}
\end{align*} satisfy the requirements of Definition \ref{def_lwf}. Indeed, $D_H$ is finite and we have $H \subset D_H d_H$ as $K \subseteq D_H d_H$ by construction. Now suppose $x' d_H =x,y' d_H=y \in H$.  We have $(x' d_H) (y' d_H) =xy \in K$. As $x'y' = (xf_{K}) (y f_{K})  =  (xy)f_{K}$, we have $(x'y')d_H = xy =(x' d_H) (y' d_H)$. 
\end{proof}

In the group case, the converse is also true, which we will demonstrate in the section on inverse semigroups.

\section{Being non-LEF}

While the previous examples of non-LEF semigroups all contain idempotents, there exist non-LEF semigroups without them.

\begin{example}\label{ex_aab}
Consider the semigroup $A = \langle a,b \mid aab = a \rangle$. We claim that it has no idempotents and it is not LEF.

Firstly note that every element of $A$ can be expressed as $b^{\beta_0} a b ^{\beta_1} \ldots a b ^{\beta_n} a^\alpha$, where $\beta_0, \alpha, n \ge 0$ and $\beta_1,\ldots,\beta_n >0$ by reducing every $aab$ to $a$. This expression is unique, as the rewriting system $aab \rightarrow a$ is noetherian (it reduces the length of the word, thus every rewriting chain is finite) and confluent (the rewriting rule cannot be applied to distinct intersecting subwords of a given word, which means that the result of the rewritings is the same regardless of their order).

Assume there is an idempotent $e$ in $A$. It has the normal form as above, so we have equation $$b^{\beta_0} a b ^{\beta_1} \ldots a b ^{\beta_n} a^\alpha b^{\beta_0} a b ^{\beta_1} \ldots a b ^{\beta_n} a^\alpha = b^{\beta_0} a b ^{\beta_1} \ldots a b ^{\beta_n} a^\alpha.$$

If $\alpha \le \beta_0$ we would get that the left hand side is equal to $$b^{\beta_0} a b ^{\beta_1} \ldots a b ^{\beta_n} a b^{\beta_0 - \alpha + 1} a b^{\beta_1} \ldots a b ^{\beta_n} a^\alpha,$$ which is clearly different from $e$.

Thus, we have $\alpha > \beta_0$ and left-hand side is equal to $$b^{\beta_0} a b ^{\beta_1} \ldots a b ^{\beta_n} a^{\alpha-\beta_0+1} b ^{\beta_1} \ldots a b ^{\beta_n} a^\alpha.$$

Similarly, if $\alpha-\beta_0+1 \le \beta_1$, we would get 
$$b^{\beta_0} a b ^{\beta_1} \ldots a b ^{\beta_n} a b ^{\beta_1 - (\alpha-\beta_0 +1) + 1 } \ldots a b ^{\beta_n} a^\alpha,$$

which is different from $e$.

Thus, we have $\alpha-\beta_0+1 > \beta_1$, and left hand-side equal to  
$$b^{\beta_0} a b ^{\beta_1} \ldots a b ^{\beta_n} a^{\alpha-\beta_0-\beta_1+2} b ^{\beta_2} \ldots a b ^{\beta_n} a^\alpha.$$

Continuing in the same manner, we must have $\alpha-\beta_0 -\ldots -\beta_{n-1}+n > \beta_n$, otherwise there will be an irreducible suffix for $e^2$ different from the one for $e$. However, the value $\eta = \alpha-\beta_0 -\ldots -\beta_{n-1}+n - \beta_n$ is the difference between the number of $a$'s and $b$'s in $e$, which is preserved under the congruence generated by the relation $a =aab$. Since we have $e^2 = e$, it follows that $2\eta = \eta$, i.e. $\eta=0$. This is a contradiction, which allows us to conclude that the initial assumption is incorrect.

Now we will prove that $A$ is not LEF. Assume it is LEF. Consider the finite subset $H = \{a,b,ab,aba \}$ of $S$ and $F_H$, $f_H$ satisfying the requirements of Definition \ref{def_lef} for $H$. Denote $p = a f_H$ and $q = b f_H$. We have $$p p q = (a f_H) (a f_H) (b f_H) = (a f_H) ((ab) f_H) =  (aab) f_H =  a f_H = p.$$ Additionally, since $F_H$ is finite, there exist $\pi, \pi'>0$ such that $p^\pi = p ^{\pi + \pi'}$. By multiplying both sides by $q^{\pi}$ on the right we get $p q= p ^{\pi'}$. This means that $p q$ commutes with $p$. However, $(p q) p = (a f_H) (b f_H) (a f_H) = ((ab) f_H)  (a f_H)  = (aba) f_H  \neq a f_H = p = p (p q)$ by the injectivity of $f_H$. Thus, we have another contradiction and $A$ is not LEF.
\end{example}

The second part of the proof in the example above can be immediately generalised to the following statement.

\begin{proposition}\label{prop_reasonaab}
Let $S$ be a semigroup such that for some $x,y \in S$ we have $xxy=x$ and $xyx \neq x$, and let $H = \{x,y,xy,xyx\}$. Then $H$ with multiplication inherited from $S$ is non-embeddable, and, in particular, $S$ is not LEF.
\end{proposition}

It has been proven that some of the natural semigroup transformations, such as adjoining a zero or taking direct products, preserve being LEF (see \cite[Section 4]{K23}). The example below demonstrates that the same does not hold true for taking the power semigroup by utilizing the non-embeddable set described above.

\begin{proposition}
The monoid $S = Mon \langle a,b \mid abab =1, baba=1 \rangle$ is an LEF semigroup, while the power semigroup $\P(S)$ is not an LEF semigroup.
\end{proposition}
\begin{proof}
Note that $S$ is a group with presentation $Gp \langle a,b\mid abab =1 \rangle$ or with an equivalent presentation $Gp\langle a,c\mid c^2 = 1\rangle$ (with a sequence of Tietze transformations $Gp\langle a,b \mid abab =1\rangle \rightarrow Gp \langle a,b,c \mid abab = 1, c =ab \rangle \rightarrow Gp \langle a,b,c \mid abab = 1, c^2 = 1, c =ab \rangle \rightarrow Gp \langle a,b,c \mid c^2 = 1, c =ab \rangle  \rightarrow Gp \langle a,b,c \mid c^2 = 1, c=ab, b =a^{-1}c \rangle \rightarrow Gp \langle a,b,c \mid c^2 = 1, b =a^{-1}c \rangle  \rightarrow Gp \langle a,c \mid c^2 = 1 \rangle$). It is straightforward to see that $Gp \langle a,c \mid c^2 = 1 \rangle$ is a free product of $\mathbb{Z}$ and  $\mathbb{Z}_2$, which is residually finite, which means that it is LEF.

For the power semigroup $\P(S)$ consider the sets $$W_a = \{\text{all the words in } \{a,b\} \text{ obtained from } a \text{ by the rewriting rule } a\rightarrow aab\},$$ $S_a = \{\text{all of the elements of } S \text{ presented by words in } W_a\}$ and $S_b = \{b\}$. We claim that $S_a S_a S_b = S_a$ while $S_a S_b S_a \neq S_a$.

To see the former, note that $S_a S_a S_b \subseteq S_a$ as for $a_1,a_2$ obtained from $a$ by using $a \rightarrow aab$ we  have that $a_1 a_2 b$ is also obtained from $a$ by using this rule, as we can initially rewrite $a$ into $aab$ and then transform the first $a$ into $a_1$ and the second $a$ into $a_2$. Additionally, any word obtained from $a$ by using  $a \rightarrow aab$ except for $a$ has the form $a_1 a_2 b$ since the first step is always $a \rightarrow aab$, and $a$ itself is presented by $aabab \in S_a S_a S_b$.

To see the latter, consider the rewriting system $abab \rightarrow 1$ and $baba \rightarrow 1$. It is noetherian and locally confluent, which means that we can obtain a normal form for each word $w$ in alphabet $\{a,b\}$.

Note that all the words in $W_A$ have form $a^{\alpha_0} b^{\beta_0} a^{\alpha_1} b^{\beta_1} \ldots a^{\alpha_n} b^{\beta_n}$ where $n \ge 0$, $\alpha_0 > \beta_0$ and $\alpha_i \ge \beta_i$ for $1 \le i \le n$ as $a \rightarrow aab$ preserves such form. Additionally, both $abab \rightarrow 1$ and $baba \rightarrow 1$ preserve such form. However, $aba$ is irreducible and does not have such a form, which means that $aba \not\in S_a$ and allows us to conclude that $S_a S_b S_a \neq S_a$. This means that $\P(S)$ is not LEF by Proposition \ref{prop_reasonaab}.
\end{proof}

Naturally, there are more non-embeddable sets.

\begin{proposition}
Let $S$ be an LEF semigroup, and $x,y$ be elements of $S$ such that $yx xy = yx$. Then either $yxy=yx$ or $(yxy)(yxx) = yx$.
\end{proposition}
\begin{proof}
Consider the finite subset $H = \{x,y,yx,yxx,yxy, (yxy)(yxx) \}$. Let $F_H$ be a finite semigroup and $f_H: H \rightarrow F_H$ satisfying the requirements of Definition \ref{def_lef} for $H$. Denote $p= x f_H$ and $q = y f_H$.

Let us demonstrate that $(qpp)^\kappa = qp^{\kappa+1}$ by induction on the power $\kappa$.

The base. For $\kappa=1$ the statement is immediate.

The step. Assume the statement holds for $k=\kappa_0$ and consider $\kappa = \kappa_0 +1$, $\kappa_0 \ge 1$. We have $$(qpp)^{\kappa_0+1} = qpp(qpp)^{\kappa_0} = qppqp^{\kappa_0+1} = qp p^{\kappa_0+1} qp^{\kappa_0+2},$$ as $qppq = (yxx f_H) (y f_H) = (yxxy) f_H = (yx)f_H = (y f_H) (x f_H)  = qp$ by the multiplication properties.

In particular, it follows that  $$q p^{\kappa+1}q =  (qpp)^\kappa y = (qpp)^{\kappa-1} qppq = (qpp)^{\kappa-1} qp = (qpp)^{\kappa-2} q pp =\ldots = q p^\kappa$$ for $\kappa \ge 1$.

As $F_H$ is finite, there exists $\kappa, \rho >0$ such that $$ qp^{\kappa+\rho+1} = (qpp)^{\kappa+\rho}=(qpp)^\kappa = qp^{\kappa+1}.$$ By multiplying it by $q^\kappa$ on the right we get $(qpp)^\rho = qp$. If $\rho=1$, then another multiplication by $q$ on the right gets us $qp = qpq$, which means $yx=yxy$ by the injectivity of $f_H$. Otherwise with the same multiplication we get $(qpp)^{\rho-1} = qpq$, meaning that $qpq$ and $qpp$ commute and $(qpq)(qpp) =(qpp) (qpq)=qppq =qp$, from which it follows that $(yxy)(yxx) = yx$.
\end{proof}

The proposition above means that any set $$H = \{x,y,yx,yxx,yxy, (yxy)(yxx) \}$$ which fails the stated property is non-embeddable.

\begin{proposition}
For the semigroup $T = Sg \langle a,b \mid baab=ba \rangle$, the semigroup $A= Sg \langle a,b \mid aab=a \rangle$ and the bicyclic monoid $B = Mon \langle a,b \mid ab=1 \rangle$ there exist elements $x,y$ such that $yx xy = yx$ and neither $yxy=yx$ nor $(yxy)(yxx) = yx$.
\end{proposition}
\begin{proof}
For the bicyclic monoid consider $x=a, y=b$. We have $baab =_B ba$, but $bab =_B b \neq_B ba$ and $bab baa = b^2 a^2 \neq_B ba$.

For $A$ and $T$ it is enough to note that the natural homomorhisms $\phi: T \rightarrow B$ with $a \phi =a, b \phi =b$ and $\psi: A \rightarrow B$ with $a \psi = a, b \psi =b$ separate $ba, bab$ and $babbaa$, while in both of them $baab = ba$.
\end{proof}

From the two previous propositions we can conclude the following.

\begin{corollary}\label{cor_tab}
The semigroups $T, A, B$ are not LEF.
\end{corollary}

Now that we know that $T = Sg \langle a,b \mid baab=ba \rangle$ is not LEF, we will also demonstrate that $T_n = Sg \langle a,b \mid (ba)(ab)^n = (ba)^n \rangle$, $n \ge 2$, are not LEF.

The following lemma is well-known, see for example \cite[Proposition 1.1]{St16}.

\begin{lemma}\label{lem_powerconsist}
Let $F$ be a finite semigroup, $s$ be an element of $F$ and $\kappa,\rho$ be such positive integers that $s^{\kappa} = s^{\kappa+\rho}$, $\kappa$ is minimal possible and $\rho$ is minimal possible for $\kappa$. Then for any $\kappa' \ge \kappa$ and $\rho' > 0$ from $s^{\kappa'}=s^{\kappa'+\rho'}$ follows $\rho' \ge \rho$.
\end{lemma}

\begin{proposition}\label{prop_prob}
Let $S$ be an LEF semigroup, $x,y$ be elements of $S$ such that $(yx)(xy)^n = (yx)^n$, $n \ge 2$. Then there exists a power $m$ such that $m<n^2+2n$, $n \nmid m$ and $(yx)^2(xy)^m =(yx)(xy)^m (yx)$.
\end{proposition}
\begin{proof}
Consider the finite subset 
\begin{align*}H = & \{x,y, xy, (xy)^2, \ldots, (xy)^{n^2+n}, yx, (yx)^2, \ldots, (yx)^{n^2+n},\\ &  (yx)(xy),\ldots, (yx)(xy)^{n^2+n},
(yx)^2(xy),\ldots, (yx)^2(xy)^{n^2+n}, \\ & (yx)(xy)(yx),\ldots, (yx)(xy)^{n^2+n}(yx)\}.
\end{align*}
Let $F_H$ be a finite semigroup and $f_H: H \rightarrow F_H$ be a map satisfying the requirements of Definition \ref{def_lef} for $H$. Denote $p= x f_H$ and $q = y f_H$. Note that by the multiplication properties and injectivity of $f_H$ we have $(qp)(pq)^n = (qp)^n$. It follows that for any $\alpha > 0$ we have $$(qp)(pq)^{n \alpha}= (qp)^n (pq)^{n (\alpha -1)} = (qp)^{2n-1} (pq)^{n (\alpha -2)}= \ldots  = (qp)^{\frac{n-1}{n}\alpha +1}.$$
 
Let $\kappa,\rho$  be such positive integers that $(pq)^{\kappa} = (pq)^{\kappa+\rho}$, $\kappa$ is minimal possible and $\rho$ is minimal possible for $\kappa$. Similarly, let $\lambda,\tau$ be such positive integers that $(qp)^{\lambda} = (qp)^{\lambda+\tau}$, $\lambda$ is minimal possible and $\tau$ is minimal possible for $\lambda$.

Note that  $(pq)^{\kappa} = (pq)^{\kappa+\rho}$ implies  $(qp)^{\kappa+1} = (qp)^{\kappa+\rho +1}$ and  $(qp)^{\lambda} = (qp)^{\lambda+\tau}$ implies  $(pq)^{\lambda+1} = (qp)^{\lambda+\tau+1}$, meaning $|\lambda-\kappa| \le 1$. Additionally, by Lemma \ref{lem_powerconsist} the equations above mean that $\rho=\tau$.

If we have $\rho=0 (\mod n)$ then we get  $(pq)^{n\kappa} = (pq)^{n\kappa+\rho}$, which in turn means that  $ (qp) (pq)^{n\kappa}  = (qp)(pq)^{n\kappa+\rho}$ and
$(qp)^{(n-1)\kappa + 1} = (qp)^{(n-1)\kappa+1+ \frac{n-1}{n}\rho}$. However,  $\frac{n-1}{n}\rho <  \rho=\tau$ which contradicts Lemma \ref{lem_powerconsist}, thus this possibility does not occur.

If we have $\rho \neq 0 (\mod n)$, then set $\kappa'$ to be $n \lceil \frac{\kappa}{n} \rceil$. We have $(qp)^{\frac{n-1}{n} \kappa' +1 } =  (qp) (pq)^{\kappa'}= (qp) (pq)^{\kappa' + n\rho}   = (qp)^{\frac{n-1}{n} \kappa' + (n-1)\rho +1 }$, which means $$(n-1) (\frac{\kappa}{n} +1) +1 > (n-1) \lceil \frac{\kappa}{n}  \rceil +1 = \frac{n-1}{n} \kappa' +1 \ge \lambda \ge \kappa-1.$$ From this follows $\kappa<n^2+n$ and $\kappa' < n^2 +n$ as well.

Set $\gamma$ to be the remainder of dividing $\kappa +\rho$ by $n$. We have $(qp)(pq)^{\kappa+n -\gamma} = (qp)(pq)^{\kappa+n -\gamma +\rho} = (qp)^{1 + \frac{n-1}{n}(\kappa+n -\gamma +\rho)}$, which means that $(qp)(pq)^{\kappa+n -\gamma}$ commutes with $qp$. Note that $\kappa+n -\gamma < n^2 +2n$ and not divisible by $n$ as $\kappa -\gamma = -\rho \neq 0 (\mod n)$. Denote $m = \kappa+n -\gamma$. By the injectivity of $f_H$ this translates to $(yx)(xy)^{m}$ commuting with $yx$, which concludes the proof.
\end{proof}

\begin{proposition}
Let $T_n$ be semigroup $Sg\langle a,b \mid (ba)(ab)^n = (ba)^n \rangle$ for $n \ge 2$. Then there exist $x,y \in T_n$ such that $(yx)(xy)^n = (yx)^n$ and $(yx)^2 (xy)^m \neq (yx)(xy)^m (yx)$ for all $m<n^2+2n$ with $n \nmid m$.
\end{proposition}
\begin{proof}
We can choose $x =a$ and $y =b$. It is immediate that $(ba)(ab)^n = (ba)^n$.

Set $m \in \NN$ such that $n \nmid m$. We will prove that $(ba)(ab)^m(ba) \neq (ba)^2(ab)^m$.

Assume that we can get from the word $w = (ba)(ab)^m(ba)$ to the word $w' = (ba)^2(ab)^m$ using a finite amount of rewritings $\rho_1 =( (ba)(ab)^n \rightarrow (ba)^n)$ or $\rho_2 = ((ba)^n \rightarrow (ba)(ab)^n)$. Denote the chain of rewritings by $w_1 = w, w_2, \ldots, w_k = w'$. Denote by $\gamma$ the maximal power of $ab$ which appears as a subword of any $w_i$.

Let $i_0$ be the smallest among indices $i$ such that $w_i$ ends with $ab$. It is possible only under rewriting $\rho_2$, meaning that we have $w_{i_0} = g_{i_0} (ba) (ab)^n$ and $w_{i_0-1} = g_{i_0} (ba)^n$ where $g_{i_0}$ is some word in $a,b$.

Let $i_1$ be the smallest among indices $i$ such that $w_i$ ends with $(ba)^n$. Evidently $i_1 \le i_0 -1$, so it is impossible to have $w_{i_1-1} = g_{i_1} (ba) (ab)^n$ and $w_{i_1} = g_{i_1} (ba)^n$ for some word $g_{i_1}$ in $a,b$. Additionally, we cannot obtain $w_{i_1}$ using $\rho_2$ as it cannot create subword $(ba)^n$ where there has not been one before. Thus, we must have $w_{i_1 -1} =  g_{i_1} (ba) (ab)^n (ba)^{j_1}$ and $w_{i_1} = g_{i_1} (ba)^{n+j_1}$ for $1 \le j_1 <  n$.

Let $i_2$ be the smallest among indices $i$ such that  $w_i$ ends with $ba (ab)^n (ba)^{j_1}$. Evidently $i_2 \le i_1 -1$, so it is impossible to have $w_{i_2-1} = g_{i_2} (ba)^{n+j_1}$ and $w_{i_2} = g_{i_2} (ba) (ab)^n (ba)^{j_1}$. Thus, we must have $w_{i_2 -1} =  g_{i_2} (ba) (ab)^{2n} (ba)^{j_1}$ and $w_{i_2} = (ba)^n (ab)^{n} (ba)^{j_1}$.

Continuing in the same manner we can get words ending with $$(ba) (ab)^{\alpha n} (ba)^{j_1}$$ for any $\alpha > 0$. Since $m$ is not divisible by $n$, which means that there always will be a previous element and potential to increase the power in the middle. However, powers of $ab$ in the chain are bound by $\gamma$, which gives us a contradiction.
\end{proof}

From the two previous propositions and Corollary \ref{cor_tab} we can infer the following.

\begin{corollary}
The semigroups $T_n$, $n \ge 1$, are not LEF.
\end{corollary}

\begin{remark}
It can be shown that the well-known Baumslag--Solitar groups $BS(n-1,n) = Gp \langle a,b \mid a b^{n-1} a^{-1} = b^{n+1} \rangle$, $n \ge 3$, fail the condition of Proposition \ref{prop_prob} for $x = a^{-1}$ and $y = ab$, meaning that they contain a non-embeddable subset  which is a not a ``group'' one, i.e. the semigroup generated by it is not a group. 
\end{remark}

A natural avenue of further study is to understand the nature of non-embeddable sets: while it is clear that expansions of non-embeddalbe sets are non-embeddable and existence of some non-embeddable sets inside a semigroup implies the existence of certain others, the classification of such sets or even deciding whether there is a finite number of ``independent" non-embeddable sets remain open problems.

\section{LEF and LWF}
While all LEF semigroups are LWF, we can demonstrate that only some of the non-LEF semigroups discussed above are non-LWF. To do this, we will require several auxiliary definitions and results.

\begin{definition}\label{def_tight}
Consider an LWF semigroup $S$ and its finite subset $H$. We say that the pair of semigroup $D_H$ and function $d_H$ satisfying the requirements of Definition \ref{def_lwf} is {\em tight} if there exists no such function $r_H: D_H \rightarrow S$ that 

\begin{enumerate}
\item $H \subseteq D_H t_H$ and for all $x', y' \in D_H$ with $x'r_H,y' r_H \in H$ it holds that $(x'y')r_H = (x'r_H)(y'r_H)$;
\item $H (r_H)^{-1} \subsetneq H (d_H)^{-1}$.
\end{enumerate}
\end{definition}

\begin{proposition}\label{prop_extight}
Let $S$ be a LWF semigroup and $H$ be a finite subset of $S$. Then there exist a semigroup $D_H$ and a map $d_H: D_H \rightarrow S$ satisfying the requirements of Definition \ref{def_lwf} such that $(D_H, d_H)$ is a tight pair.
\end{proposition}
\begin{proof}
Consider a finite semigroup $D_H$ and a map $d'_H: D_H \rightarrow S$ satisfying the requirements of Definition \ref{def_lwf}. If the pair $D_H, d'_H$ is not tight, there exists $r_H$ such that the pair $D_H,r_H$ also satisfies the requirements of the definition and  $H (r_H)^{-1} \subsetneq H (d'_H)^{-1}$. Note that $H (d_H)^{-1}$ is finite as a subset of a finite structure $D_H$. If $D_H,r_H$ is not tight, we can continue this process inductively with finding another function $r'_H$  such that $D_H,r'_H$ also satisfies the requirements of the definition and  $H (r'_H)^{-1} \subsetneq H (r_H)^{-1}$, and so on. Due to the fact that the size of the pre-image of $H$ cannot be smaller than $|H|$, the process will stop eventually, resulting in a tight pair $D_H, d_H$.
\end{proof}

We can establish a stricter structural property for tight pairs.

\begin{definition}
Let $S$ be an infinite LWF semigroup, $H = \{h_1,\ldots,h_t\}$ be its finite subset, $D_H,d_H$ be a tight pair and $\{h'_1,\ldots,h'_t\}$ be a set of elements of $D_H$ such that $h_i = h'_i d_H$. We say that a product $h'_{i_1} \ldots h'_{i_k}$, $k \ge 1$, $i_1,\ldots, i_k \in \{1,\ldots,t\}$ is {\em accurate} if

\begin{enumerate}
\item $h'_{i_1} \ldots h'_{i_k} \in H (d_H)^{-1}$;

\item One of the following holds:

\begin{itemize}
\item[a.] $k=1$;
\item[b.] There exists an index $j$, $1 \le j <k$ such that $h'_{i_1} \ldots h'_{i_j}$ is accurate and $h'_{i_{j+1}} \ldots h'_{i_k}$ is accurate.
\end{itemize}
\end{enumerate}
\end{definition}

\begin{proposition}\label{prop_tight}
Let $S$ be an infinite LWF semigroup, $H = \{h_1,\ldots,h_t\}$ be its finite subset, $D_H,d_H$ be a tight pair and $\{h'_1,\ldots,h'_t\}$ be a set of elements of $D_H$ such that $h_i = h'_i d_H$. Then  for every element $w' \in H (d_H)^{-1}$  there exist a product $h'_{i_1} \ldots h'_{i_k}$ equal to $w'$, $k \ge 1$, $i_1,\ldots, i_k \in \{1,\ldots,t\}$ which is accurate.
\end{proposition}

\begin{proof}
Choose an arbitrary $s \in S \setminus H$ and denote by $U$  the subset of $H (d_H)^{-1}$ consisting of elements which cannot be presented with an accurate product. Assume that $U$ is non-empty.

Define the function $r_H: D_H \rightarrow S$ as follows:
\begin{align*} x' r_H & = \begin{cases} x' d_H, & \text{if }x' \not\in U, \\ s, & \text{if } x' \in U.  \end{cases}
\end{align*}
Our goal is to demonstrate that $D_H, r_H$ satisfy the properties of Definition \ref{def_lwf}.

Firstly, since $h'_i r_H = h'_i d_h = h_i$ and $h'_i \not\in U$, we have $H \subset D_H r_H$. Secondly, for all $x', y' \in D_H$ with $x'r_H,y' r_H \in H$ we have $x'r_H =x' d_H ,y' r_H = y' d_H$ and $x'y'$ is either not in  $H (d_H)^{-1}$, meaning that it is outside $U$ as well and $(x'y')r_H = (x' y') d_H$, or $x'y'$ is in $H (d_H)^{-1}$ but not in $U$ as $x',y'$ not in $U$, meaning that the product of the accurate products representing them (which is accurate itself) equals to $x'y'$, allowing us to conclude that $(x' y') r_H =  (x' y') d_H$. Thus we have $(x'r_H)(y' r_H) = (x'd_H)(y' d_H) = (x'y' )d_H= (x'y' )r_H$.

Thus, the pair $D_H,d_H$ is not tight as  $H (r_H)^{-1} = H (d_H)^{-1} \setminus U \subsetneq H (d_H)^{-1}$. The resulting contradiction means that our assumption was incorrect and $U$ is empty.
\end{proof}

\begin{proposition}
The bicyclic monoid $B = \langle a,b \mid ab= 1\rangle$ is not an LWF semigroup.
\end{proposition}
\begin{proof}
Assume that $B$ is LWF. Consider the set $H = \{ 1,a,b,ba\}$ and consider a finite semigroup $D_H$ and a map $d_H: D_H \rightarrow B$ satisfying the requirements of Definition \ref{def_lwf} such that $D_H, d_H$ is a tight pair.

Let $x',y'$ be arbitrary elements of $D_H$ such that $x' d_H = a$, $y' d_H = b$. We have $(x'y') d_H = (x' d_H) (y' d_H) = a b =1$, by the property of $d_H$ and by the same property any power of $x' y'$ maps to $1$.

As $D_H$ is finite, there exists a power $n$ such that $(x' y')^n$ is idempotent.

Consider elements $a' = (x' y')^{2n-1} x'$ and $b' = y' (x' y')^{n}$. By the multiplicative property we have $a'd_H = a$ and $b' d_H =b$. Moreover, $a' b' a' =  (x' y')^{2n-1} x'  y' (x' y')^{n}  (x' y')^{2n-1} x' = (x' y')^{2n-1} x'  = a'$ and similarly $b' a' b' = b'$.

We want to demonstrate that $a' b'$ acts as an identity for the subsemigroup $\langle a', b'\rangle$ of $D_H$, while $b' a'$ is not an identity for this subsemigroup. The latter follows from the former and the fact that $a' b' \neq b' a'$ as they have different images ($1$ and $ba$, respectively) under $d_H$.

To prove the former, consider $a'$. By Proposition \ref{prop_tight} applied to set of pre-images $\{a'a'b'b', a'a' b', b', b'  a'a' b'\}$, there exists an accurate product of these pre-images equal to $a'$. Since $b' a' b' = a' b'$,  $a' b'$ acts as an identity on the right for each of these pre-images, meaning that the same is true for their product $a'$, i.e. $a' a' b' = a'$.

Similarly, $ a' b' b' = b'$. As we already know that $a' b' a' = a'$ and  $b' a' b' = b'$, $a' b'$ is indeed an identity for $\langle a', b'\rangle$. 

Thus, by \cite[Lemma 1.31]{ClP} the semigroup $\langle a', b'\rangle$ is isomorphic to $B$, meaning that the semigroup $D_H$ cannot be finite. This contradicts our initial assumption.
\end{proof}

\begin{proposition}
The semigroup $A = \langle a,b \mid aab = a\rangle$ is not LWF.
\end{proposition}
\begin{proof}
Assume that $A$ is LWF. Consider its subset $H = \{a,b,ab,aba \}$ and a finite semigroup  $D_H$ and a map $d_H: D_H \rightarrow A$ satisfying the requirements of Definition \ref{def_lwf} for $H$ such that $D_H, d_H$ is a tight pair. 

Our goal is to prove that for all $x,y \in H$ such that $xy \in H$ we have $Q_x Q_y = Q_{xy}$, where $Q_x, Q_y$ and $Q_{xy}$ are the sets of pre-images of $x,y,xy$ under $d_H$ and the product on the left-hand side is taken in the sense of the power semigroup $\P(D_H)$. Once we have established that, it will follow that $F_H = \P(D_H)$ and $f_H: H \rightarrow F_H$ defined by $h \mapsto Q_h$ satisfy the requirements of Definition \ref{def_lef} as $F_H$ is finite, $Q_h$ are distinct for different $h$ and the multiplication property is exactly described above. However, by Proposition \ref{prop_reasonaab} it is impossible to find such $F_H$ and $f_H$.

Consider the multiplication table of the elements of $H$.

\begin{tabular}
{| c || c | c | c | c |}
\hline
          & $a$ & $b$ & $ab$ & $aba$ \\
\hline \hline
$a$ & $a^2$ & $ab$ & $a$ & $a^2$ \\
\hline
$b$ & $ba$ & $b^2$ & $bab$ & $baba$ \\
\hline
$ab$ & $aba$ & $abb$ & $abab$ & $ababa$ \\
\hline
$aba$ & $aba^2$ & $abab$ & $aba$ & $aba^2$ \\
\hline
\end{tabular}

Fix an arbitrary pre-image $a'$ of $a$ and $b'$ of $b$ under $d_H$. By the multiplication property, $a'a'b'$ is a pre-image of $a$, $a' b'$ is a pre-image of $ab$ and $a'b'a'$ is a pre-image of $aba$ under $d_H$ as well.

First, let us prove that $Q_a Q_b = Q_{ab}$. From the multiplication property we have that $Q_a Q_b \subseteq Q_{ab}$. To see the reverse inclusion, consider $z' \in Q_{ab}$. By Proposition \ref{prop_tight} there exists an accurate product of elements in $\{a',b',a'b',a'b'a'\}$ equal to $z'$. If the product consists of one term, i.e. $z' = a'b'$, we have $z' \in Q_a Q_b$. Otherwise $z' = u' v'$, where both $u'$ and $v'$ are accurate products, so in particular $u'd_H, v'd_H \in H$ and by the multiplication property $(u'd_H) (v'd_H) = z'd_H =ab$. Following the multiplication table, this means that $u' d_H = a$ and $v' d_H = b$, and $z' \in Q_a Q_b$ again. Thus, $Q_a Q_b \supseteq Q_{ab}$ and, furthermore, $Q_a Q_b = Q_{ab}$.

Second, let us prove that $Q_{a} Q_{ab} = Q_{a}$. From the multiplication property we have that $Q_a Q_{ab} \subseteq Q_{a}$. To see the reverse inclusion, consider $z' \in Q_{a}$. By Proposition \ref{prop_tight} there exists an accurate product of elements in the set $\{a'a'b, b' ,a'b', a'b'a'\}$ equal to $z'$. If the product consists of one term, i.e. $z' = a'a'b'$, we have $z' \in Q_a Q_{ab}$. Otherwise $z' = u' v'$, where both $u'$ and $v'$ are accurate products, so in particular $u'd_H, v'd_H \in H$ and by the multiplication property $(u'd_H) (v'd_H) = z'd_H =a$. Following the multiplication table, this means that $u' d_H = a$ and $v' d_H = ab$, and $z' \in Q_a Q_{ab}$ again. Thus, $Q_a Q_{ab} \supseteq Q_{a}$ and, furthermore, $Q_a Q_{ab} = Q_{a}$.

Finally, we need to prove that $Q_{ab} Q_a = Q_{aba} = Q_{aba} Q_{ab}$. Note that it follows from Proposition \ref{prop_tight} that $Q_{aba} = Q_{ab} Q_{a} \cup Q_{aba} Q_{ab}$ similarly to the above. We will show $Q_{ab} Q_a = Q_{aba} Q_{ab}$ from which the initial equality will follow.

Assume there exist $w' \in Q_{aba} \setminus Q_{ab}Q_{a}$ (in particular, $w' \neq a'b'a'$). Denote this property by $(*)$. By Proposition \ref{prop_tight} we know that $w' = u' v'$ where both $u'$ and $v'$ are in $H d_H^{-1}$, which leaves us only the possibility $u' \in Q_{a'b'a'}$ and $v' \in Q_{a'b'}$. Note that $u'$ must satisfy $(*)$ as well as otherwise if $u' = s' t'$ where $s' \in Q_{ab}$ and $t' \in Q_a$ we have $w' = s' (t' v')$ with $t'v' \in Q_{a} Q_{ab} = Q_{a}$.

Let us demonstrate that $D_H, d_H$ is not tight by constructing another function $r_H$ as follows: 
\begin{align*}
 x' r_H & = \begin{cases} x' d_H, & \text{if }x' \text{ does not satisfy } (*), \\ a^2, & \text{if } x' \text{ satisfies } (*).  \end{cases}
\end{align*}

Our goal is to demonstrate that $D_H, r_H$ satisfies the properties of Definition \ref{def_lwf}.

Since $(*)$ is only applicable to elements of $Q_{aba}$ and $a'b'a'$ does not satisfy $(*)$ we have $D_H r_H \supseteq H$. Additionally, for all $x', y' \in D_H$ with $x'r_H,y' r_H \in H$ we have $x'r_H =x' d_H ,y' r_H = y' d_H$, and since $x'$ and $y'$ do not satisfy $(*)$, $x'y'$ does not satisfy $(*)$ as well according to the observation above. Thus $(x' y') r_H =  (x' y') d_H$ and we have $$(x'r_H)(y' r_H) = (x'd_H)(y' d_H) = (x'y' )d_H= (x'y' )r_H.$$ This means that our assumption is incorrect and no element of $Q_{aba}$ satisfies $(*)$, i.e. $Q_{ab} Q_{a} \supseteq Q_{aba} Q_{ab}$. On the other hand, $Q_{ab} Q_{a} = Q_{ab} Q_{a} Q_{ab} \subseteq Q_{aba} Q_{ab}$. Thus, $Q_{aba}Q_{ab} = Q_{ab}Q_{a} = Q_{aba}$.

This together with the initial observation concludes the proof.

\end{proof}

It remains an open question whether all LWF semigroups are LEF.

\section{Inverse semigroups}

The class of inverse semigroups can be seen as an intermediate between general semigroups and groups.

\begin{definition}
A semigroup $S$ is called {\em inverse} if for every element $x \in S$ there exists a unique element $y \in S$ such that $xyx = x$ and $yxy = y$. This element $y$ is denoted by $x^{-1}$.
\end{definition}

One of the classical examples of the inverse semigroups are symmetric inverse monoids (semigroups) which are the sets of partial bijections on a given set $\Sigma$ with semigroup operation being the composition. 

The structural theorem for the inverse semigroups similar to the Cayley Theorem in the group case is known as Wagner--Preston theorem.

\begin{theorem}\cite[Theorem 5.1.7]{How95}
Let $S$ be an inverse semigroup. Then there exists a symmetric inverse semigroup $I_X$ and a monomorphism $\phi$ from $S$ into $I_X$.
\end{theorem}

Similarly to the group and semigroup cases, we can define local embeddability into finite for the class of inverse semigroups.

\begin{definition}\label{def_ilef}
An inverse semigroup $S$ is called {\em locally embeddable into the class of inverse finite semigroups} (an {\em iLEF} semigroup  for short) if for every finite subset $H$ of $S$ there exists a finite inverse semigroup $F^{(i)}_H$ and an injective function $f^{(i)}_H: X \rightarrow F^{(i)}_H$, such that $\forall x,y \in H$ with $xy \in H$ it holds that $(xy)f^{(i)}_H = (xf^{(i)}_H)(yf^{(i)}_H)$.
\end{definition}

\begin{proposition}\label{prop_lefisem}
An inverse semigroup is an LEF semigroup if and only if it is an iLEF semigroup.
\end{proposition}
\begin{proof}
Let $S$ be an iLEF semigroup. It is immediate that it is an LEF semigroup as well.

Let $S$ be an inverse semigroup which is LEF. Let $H$ be a finite subset of $S$. We can expand $H$ to the set $K = H \cup H^{-1} \cup H H^{-1} \cup H^{-1} H$.

Consider the finite semigroup $F_{K^3}$ and function $f_{K^3}$ satisfying Definition \ref{def_lef} for $K^3$. Without the loss of generality we can consider $F_{K^3}$ to be a full transformation semigroup $T_\Sigma$ of a finite set $\Sigma$ as we can embed $F_{K^3}$ in $T_\Sigma$ and carry $f_{K^3}$ over to the bigger semigroup. We will  shorten $f_{K^3}$ to $f$ for brevity in the rest of the proof. 

Now set $F^{(i)}$ to be the inverse semigroup consisting of all partial bijections on $\Sigma$ (the symmetric inverse monoid on $\Sigma$) and $f^{(i)}$ to be a map from $K^3$ to $F^{(i)}$ such that for $x \in K^3$ the image  $x f^{(i)}$ is a partial bijection between $\Im x^{-1}f$ and $\Im x f$ induced by restricting $x f$ on $\Im x^{-1}f$. We will prove that $F^{(i)}, f^{(i)}$ satisfy the conditions of Definition \ref{def_ilef} for the set $K$ (and, in turn, for $H \subseteq K$).

Firstly, we need to demonstrate that $f^{(i)}$ is well-defined, meaning that restricting $x f$ on $\Im x^{-1}f$ provides a bijection between $\Im x^{-1}f$ and $\Im x f$, i.e. $f^{(i)}$ sends elements of $K$ to the inverse semigroup $F^{(i)}$.

Take arbitrary $a,b \in \Im x^{-1}f$ such that $a \neq b$. We claim that  $a (xf) \neq b (xf)$. By definition there exist $c,d \in \Sigma$ such that $c (x^{-1}f) = a$, $d (x^{-1}f) =b$. Note that $x^{-1} x x^{-1} =x^{-1}$, correspondingly $(x^{-1}f) (xf) (x^{-1}f) =x^{-1}f$. We get
\begin{align*}
(a (xf)) (x^{-1}f) = c ((x^{-1}f) (xf) (x^{-1}f))= c (x^{-1}f) = a \neq & \\ b = d (x^{-1}f) = d ((x^{-1}f) (xf) (x^{-1}f)) = (b (xf)) (x^{-1}f), &
\end{align*}
 meaning that $a (xf)$ and $b(xf)$ are separated by the function $(x^{-1}f)$, i.e. distinct. Thus, the restriction of $xf$ to $\Im x^{-1}f$ is an injective map.

To prove that this restriction is surjective, note that $x x^{-1} x =x$, correspondingly  $(xf) (x^{-1}f) (xf) =xf$. This, in particular, means that the image of the restriction of $xf$ to $\Im (xf) (x^{-1}f)$ coincides with the image of $xf$, and, since $\Im (xf) (x^{-1}f) \subset \Im x^{-1}f$, we get that the image of the restriction of $xf$ to $\Im x^{-1}f$ indeed coincides with $\Im xf$ as well.

Note that for $e \in K$ such that $e=e^2$ (and correspondingly $e = e^{-1}$) we have $e f^{(i)}$ to be the identity function on $\Im e f$. This means for $x \in K$ we have $(xf^{(i)})^{-1} = x^{-1} f^{(i)}$ as $x^{-1} f^{(i)}$ is by construction the bijection between $\Im xf$ and $\Im x^{-1}f$ which is inverse to $x f^{(i)}$ due to $(x f)(x^{-1} f) =(xx^{-1}) f$ and $(x^{-1} f)(x f) =(x^{-1}x) f$  being identity functions on $\Im x^{-1}f$ and $\Im xf$ respectively.

Secondly, we need to prove that $f^{(i)}$ is injective on $K$. Assume that for $x,y \in H$ we have $xf^{(i)} = yf^{(i)}$. This means that both $hf^{(i)}$ and $yf^{(i)}$ provide the same partial bijection of $\Sigma$ (in particular, $\Im xf = \Im yf$ and $\Im x^{-1}f = \Im y^{-1}f$), and, by the observation above, $x^{-1}f^{(i)}$ and $y^{-1}f^{(i)}$ provide the same inverse partial bijection too. Consider the element $(xf)(y^{-1}f)(xf)$. Note that by construction $(xf)(x^{-1}f) = (xf)(x^{-1}f^{(i)})$, by our assumption  $(xf)(x^{-1}f^{(i)}) = (hf)(y^{-1}f^{(i)})$ and finally $(xf)(y^{-1}f^{(i)}) =(xf)(y^{-1}f)$ by construction again. This means that $$(xy^{-1}x)f=(xf)(y^{-1}f)(xf) =(xf)(y^{-1}f)(xf) = xf,$$ and as $x,y, xy^{-1}, xy^{-1}x \in K^3$, we get $xy^{-1}x= x$ by Definition \ref{def_lef}. Similarly we can get $y^{-1} x y^{-1} = y^{-1}$, which allows us to conclude that $x^{-1}= y^{-1}$ and $x = y$.

Finally, we need to prove that $\forall x,y \in K$ with $xy \in K$ it holds that $(xy)f^{(i)} = (xf^{(i)})(yf^{(i)})$. 

By construction, we have $xf^{(i)}$ to be the bijection between $\Im x^{-1}f$ and $\Im xf$ induced by $xf$, $yf^{(i)}$ to be the bijection between $\Im y^{-1}f$ and $\Im yf$ induced by $yf$, and $(xy)f^{(i)}$ to be the bijection between $\Im ((xy)^{-1})f$ and $\Im (xy)f$ induced by $(xy)f$. Note that $\Im ((xy)^{-1})f = \Im (y^{-1}x^{-1})f = \Im [(y^{-1}f)(x^{-1}f)]$. This image is a subset of $\Im x^{-1}f$, meaning that $xf^{(i)}$ acts on it exactly as $xf$. Additionally, as $(y^{-1}f)  (x^{-1}f) (xf) = (y^{-1}f) (x^{-1}f) (xf)  (yf) (y^{-1}f) $ we have $\Im[ (y^{-1}f)  (x^{-1}f) (xf) ]\subset \Im y^{-1}f$, meaning that $yf^{(i)}$ acts on the former exactly as $yf$. This allows us to conclude that for $a \in \Im ((xy)^{-1})f$ we have 
\begin{align*}
a((xy)f^{(i)}) & = a ((xy)f) = a((xf)(yf)) = (a(xf))(yf) = (a(xf))(yf^{(i)})  \\ & = (a(xf^{(i)}))(yf^{(i)}) = a((xf^{(i)})(yf^{(i)})),
\end{align*}
i.e. $(xy)f^{(i)}$ is the restriction of $(xf^{(i)})(yf^{(i)})$ onto  $\Im ((xy)^{-1})f$. Since the function $(xf^{(i)})(yf^{(i)})$ is a bijection and the size of its image is less than or equal to $|\Im (xf)(yf)| = |\Im (xy)f|=|\Im (xy)f^{(i)}|$, this means that  $(xy)f^{(i)} = (xf^{(i)})(yf^{(i)})$.
\end{proof}

We can also apply the notion of being LWF to inverse semigroups in a similar manner.

\begin{definition}\label{def_ilwf}
An inverse semigroup $S$ is called {\em locally wrapped by the class of inverse finite semigroups} (an {\em iLWF} semigroup  for short) if for every finite subset $H$ of $S$ there exists a finite  inverse semigroup $D_H$ and a function $d_H: D_H \rightarrow S$, such that $H \subset D_H d_H$ and for all $x', y' \in D_H$ with $x'd_H,y' d_H \in H$ it holds that $(x'y')d_H = (x'd_H)(y'd_H)$.
\end{definition}

In order to establish a partial correspondence between iLEF and iLWF semigroups we require the following results.

\begin{remark}
As the proofs of Proposition \ref{prop_extight} and Proposition \ref{prop_tight} do not change the wrapping semigroup $D_H$, we can apply them to iLWF case as well.
\end{remark}

For the next set of lemmas we will consider {\em symmetrised} subsets $K$ of inverse semigroups, i.e. such that $K = K^{-1}$ and $K \supseteq K K^{-1}$, $K \supseteq K^{-1} K$. Note that for any given finite $H$, the set $K = H \cup H^{-1} \cup H H^{-1} \cup H^{-1} H$ is symmetrised. We will also limit ourselves to infinite semigroups as finite case is trivial (all structures are LEF/LWF).

\begin{lemma}\label{lem_tightinv}
Let $S$ be an infinite iLWF semigroup, $K$ be its finite symmetrised subset, and $D_{K}, d_{K}$ be a tight pair for $K$. Then for all $w' \in K (d_{K})^{-1}$ holds $(w' d_{K})^{-1} = w^{\prime -1} d_{K}$.
\end{lemma}
\begin{proof}
We will denote $D_{K}$ and $d_{K}$ as $D$ and $d$ for brevity.

Consider an arbitrary $h \in K$ which is not an idempotent. There exists $x' \in D$ such that $x' d = h$. Since $K$ is symmetrised, $h^{-1}$ is also in $K$, which means that  there exists $y' \in D$ such that $y d = h^{-1}$. Since $D$ is finite, there exists a positive $n$ such that $(x' y')^n$ is a idempotent. Additionally, $x' y'$ maps to $hh^{-1}$ under $d$ by the multiplication property, which means $(uv)^m$ for any positive $m$ maps to $hh^{-1}$ as well. Thus, by the same property, $u' = (x'y')^n x' $ maps to $h$ and $v'=y' (x'y')^{2n-1}$ maps to $h^{-1}$. We claim that $u'$, $v'$ are inverse to each other. To see this, note that $$u' v' u' = (x'y')^n x' y' (x'y')^{2n-1} (x'y')^n x'  = (x'y')^{4n} x' = (x'y')^n x' = u'$$ and $$v ' u' v' = y' (x'y')^{2n-1} (x'y')^n x' y' (x'y')^{2n-1} = y' (x'y')^{5n-1} = y' (x'y')^{2n-1} = v'.$$ Thus, we can find inverse pre-images for $h$ and $h^{-1}$ under $d$.

Consider an arbitrary $h \in K$ which is an idempotent. There exists $z' \in D$ such that $z' d = h$. Since $D$ is finite, there exists a positive $n$ such that $(z')^n$ is a idempotent and by multiplicative properties $(z')^n$ maps to $h$ under $d$. Thus, we can find an idempotent pre-image for idempotent $h \in K$ under $d$.

Now let us apply Proposition \ref{prop_tight} for the set of pre-images $K' = \{h'_1,\ldots,h'_t\}$  such that if $h_i$ is idempotent then $h'_i$ is also idempotent and if $h_i = h_j^{-1}$ then we have $h'_i = h_j^{\prime -1}$ (such pre-images exist by the arguments above). Now consider $w' \in K d^{-1}$. By the proposition  there exists an accurate product $h'_{i_1} \ldots h'_{i_k}$ equal to $w'$. To finish the proof we will use an induction on $k$ in order to demonstrate that $(w' d)^{-1} = w^{\prime -1} d$.

The base. For $k=1$ the statement follows immediately from our choice of $K'$.

The step. Assume the statement holds for $k=1,\ldots,k_0-1$, $k_0 \ge 2$ and consider $k=k_0$. As the product $h'_{i_1} \ldots h'_{i_{k_0}}$ is accurate and $k_0 \ge 2$, there exists an index $j$, $1 \le j <k_0$ such that $h'_{i_1} \ldots h'_{i_j}$ is accurate and $h'_{i_{j+1}} \ldots h'_{i_{k_0}}$ is accurate. Denote the value of the first product by $u'$ and the value of the second product by $v'$. It holds that $u'v' =w'$, $w' d = (u' d) (v' d)$, and by the induction hypothesis we have $(u' d)^{-1} = u^{\prime -1} d$ and $(v' d)^{-1} = v^{\prime -1} d$. Moreover, $u' d, v' d, w' d \in K$, meaning that $(u' d)^{-1}, (v' d)^{-1}, (w' d)^{-1} \in K$, which allows us to check 
\begin{align*}
(w' d)^{-1} & = ((u' v') d)^ {-1} =   ((u' d) (v' d))^ {-1} = (v' d)^{-1} (u' d)^{-1} \\ & = (v^{\prime -1} d) (u^{\prime -1} d) = (v^{\prime -1} u^{\prime -1}) d = w^{\prime -1} d. 
\end{align*}
\end{proof}

\begin{definition}
Let $S$ be an infinite iLWF semigroup, $H$ be its finite subset and $D_H, d_H$ be a tight pair for $H$. We say that $h' \in h d_H^{-1}$ is {\em h-minimal}, $h  \in H$, if  $h' h^{\prime -1} \le h'' h^{\prime \prime -1}$ for any $h'' \in  h d_H^{-1}$ in the natural partial order on $D_H$.
\end{definition}

\begin{lemma}\label{lem_hmininv}
Let $S$ be an infinite iLWF semigroup, $K$ be its finite symmetrised subset, $D_K, d_K$ be a tight pair for $K$ and $h' \in D_K$ be an $h$-minimal element for some $h \in K$. Then $h^{\prime -1}$ is $h^{-1}$-minimal.
\end{lemma}
\begin{proof}
By Wagner--Preston theorem $D_K$ embeds into the symmetric inverse monoid on some finite set $\Sigma$.

Consider arbitrary $h'' \in h^{-1} d_K^{-1}$. Assume that $h^{\prime -1} h'  > h'' h^{\prime \prime -1}$, i.e. $\Sigma_1 \supsetneq \Sigma_2$ where $\Sigma_1$ and $\Sigma_2$ are the subsets of $\Sigma$ such that the first and the second respectively correspond to the identity maps on them. Note that it means that $h' = h' h^{\prime -1} h' > h' h'' h^{\prime \prime -1}$ as the left-hand side is a partial bijection between sets of order $|\Sigma_1|$ and the right-hand side is a partial bijection between sets of order no more than $|\Sigma_2|$, so they cannot be equal.

The element  $h' h'' h^{\prime \prime -1}$ maps to $h$ under $D_K$ by the multiplicative rules, however it follows that $(h' h'' h^{\prime \prime -1})(h' h'' h^{\prime \prime -1})^{-1} < h' h^{\prime -1}$, meaning that $h'$ is not $h$-minimal. This means our assumption is incorrect, and $h^{\prime -1}$ is $h^{-1}$-minimal.
\end{proof}

\begin{lemma}\label{lem_min}
Let $S$ be an infinite iLWF semigroup, $K$ be its finite symmetrised subset and $D_K, d_K$ be a tight pair for $K$. Then for every $h \in K$   there exists an $h$-minimal element of $D_K$.
\end{lemma}

\begin{proof}
Consider the set $E' = (h h^{-1}) d_K^{-1}$. Denote by $e'$ the product of all of the elements inside $E'$. By the multiplication properties, $e' \in E'$. 

Now consider an arbitrary element $h' \in h d_K^{-1}$. The element $e' h'$ will be $h$-minimal as for any   $h'' \in  h d_K^{-1}$ we have $h'' h^{\prime \prime -1} \ge e' = e' h' h^{\prime -1} = (e' h') (e'h')^{-1}$.
\end{proof}

\begin{lemma}\label{lem_min2}
Let $S$ be an infinite iLWF semigroup, $K$ be its finite symmetrised subset, $D_K, d_K$ be a tight pair for $K$ and $h'$ be an $h$-minimal element for $h \in H$. Then $h' h^{\prime -1}$ is a $h h^{-1}$ minimal element and $h^{\prime -1} h'$ is a $h^{-1}h$-minimal element.
\end{lemma}

\begin{proof}
We will demonstrate that $h' h^{\prime -1}$ is a $h h^{-1}$ minimal element with the other proof being similar.

Let $f'$ be an arbitrary element in $(h h^{-1}) d_K^{-1}$. By the multiplication properties $f' h'$ belongs to $h d_K^{-1}$. Since $h'$ is $h$-minimal, we get $h' h^{\prime -1}  \le (f' h') (f' h')^{-1}$. This allows us to conclude that  $(h' h^{\prime -1})  (h' h^{\prime -1})^{-1} =   h' h^{\prime -1}  \le (f' h') (f' h')^{-1} = f' h' h^{\prime -1} f^{\prime -1} \le f f^{\prime -1}$, which means that $h' h^{\prime -1}$ is a $h h^{-1}$ minimal.
\end{proof}

\begin{lemma}\label{lem_unidem}
Let $S$ be an infinite iLWF semigroup, $K=\{h_1\ldots,h_t\}$ be its finite symmetrised subset  and $D_K, d_K$ be a tight pair. Then for any $e \in E(S) \cap K$ there is a unique idempotent in $e d_K^{-1}$. 
\end{lemma}

\begin{proof}
By Wagner--Preston theorem $D_K$ embeds into the symmetric inverse monoid on $\Sigma$. 

Assume that there exists $e \in E(S) \cap H$ such that there are idempotents $e',f' \in e d_K^{-1}$ with $e' < f'$ in the natural partial order on $D_H$.

Consider the set $\Sigma_1 \subseteq \Sigma$ which is the domain/range of $e'$ as a partial bijection on $\Sigma$ and the set $\Sigma_2 \subseteq \Sigma$ which is the domain/range of $f'$ as a partial bijection on $\Sigma$. Naturally, $\Sigma_1 \subsetneq \Sigma_2$. Let us expand the set $\Sigma$ by a new set $\Omega$, which is in a 1-to-1 correspondence to elements of $\Sigma_2 \setminus \Sigma_1$. Denote this correspondence by $\eta: \Sigma_1\setminus \Sigma_2 \rightarrow \Omega$.

We will construct a new function $d^*_K$ from the symmetric inverse monoid on $\Sigma \sqcup \Omega$, denoted $I$, such that $I$ and $d^*_K$ satisfy Definition \ref{def_ilwf}. 

Let us define a map $\gamma$ from $D_K$ to $I$. Consider an arbitrary element $x'$ of $D_K$. Let $\Delta$ and $\Lambda$ denote its domain and range as a partial bijection on $\Sigma$. We will construct the element $x' \gamma$ of $I$ by specifying its range and image $\Delta'$ and $\Lambda'$ and the connection between them.

Assume $\Delta \cap  (\Sigma_2 \setminus \Sigma_1) \neq \emptyset$ and $(x' x^{\prime -1} e') d_K \neq (x' x^{\prime -1}) d_K$. In this case set $\Delta '$ to be $[\Delta  \setminus (\Sigma_2 \setminus \Sigma_1) ] \cup  (\Delta  \cap  (\Sigma_2 \setminus \Sigma_1) )\eta $. Otherwise set $\Delta' = \Delta $.

Similarly, assume $\Lambda \cap  (\Sigma_2 \setminus \Sigma_1) \neq \emptyset$ and $(e' x^{\prime -1} x') d_K \neq ( x^{\prime -1} x') d_K$. In this case set $\Lambda'$ to be $[\Lambda' \setminus (\Sigma_2 \setminus \Sigma_1) ] \cup  (\Lambda' \cap  (\Sigma_2 \setminus \Sigma_1) ) \eta$. Otherwise set $\Lambda' = \Lambda$.

Finally, let $x' \gamma$ be the map between $\Delta'$ and $\Lambda'$ which maps elements exactly as $x'$, except for changing all $b \in (\Sigma_2 \setminus \Sigma_1)$ from domain/range  into $b\eta$. Define $(x' \gamma) d^*_K := x' d_K$. Note that $d^*_K$ is well-defined as $\gamma$ is a 1-to-1 map (we can recover $x'$ from $x' \gamma$ by replacing all of $b \eta$ with $b$ and keeping the relations).

Let $K'$ be the set $\{h'_1,\ldots,h'_t\}$ such that each $h'_i$ is $h_i$-minimal, if $h_i$ is idempotent then $h'_i$ is also idempotent and if $h_i = h_j^{-1}$ then $h'_i = h_j^{\prime -1}$. We claim that $\gamma$ is an injective homomorphism from $T = \langle K' \rangle$ into $I$.

Consider an $h$-minimal element $h'$ from $K'$ and its domain $\Upsilon_1$ and range $\Upsilon_2$ as a partial bijection on $\Sigma$. Note that if $\Upsilon_1 \cap  (\Sigma_2 \setminus \Sigma_1) \neq \emptyset$ then it is impossible for the equality $(h' h^{\prime -1} e') d_K = (h' h^{\prime -1}) d_K = h h^{-1}$ to hold since $(h' h^{\prime -1} e') (h' h^{\prime -1} e') ^{-1} = h' h^{\prime -1} e' < h' h^{\prime -1} = h' h^{\prime -1}  (h' h^{\prime -1} )^{-1}$ and $h' h^{\prime -1}$ is an $h h^{-1}$-minimal element by Lemma \ref{lem_min2}. Thus, $\gamma$ acts on $K'$ exactly by renaming all of the elements of the set $\Sigma_2 \setminus \Sigma_1$ in the domains and ranges to the respective elements of $\Omega$ under $\eta$, which evidently preserves the semigroup structure and has the injective property.

Note that $f'$ maps to $f'$ under $\gamma$ as $f' f^{\prime -1} e' = e'$ and $(f' f^{\prime -1} e') d_K = e' d_K$. This means that $f'$ does not belong to $T \gamma$ as none of the generators in the set $K' \gamma$ contain elements of $\Sigma_2 \setminus \Sigma_1$ in their domains or ranges. However, $f' \in T$ as by Proposition \ref{prop_tight} applied to the set $K'$ of pre-images of $K$ as $D_K, d_K$ is tight. This is a contradiction, meaning that there are no idempotents $e',f' \in e d_K^{-1}$ with $e' < f'$ in the natural partial order inside $D_K$. 

This immediately implies that there are no idempotents  $e',f' \in e d_K^{-1}$ with $e' \neq f'$ as for such a pair we have $e'f' \in e d_K^{-1}$, $e'f'$ is an idempotent and at least one of $e'f' <e'$ and $e'f' < f'$ holds.
\end{proof}

Now we can demonstrate the final result of this section.

\begin{theorem}
Let $S$ be a countable iLWF semigroup with a finite number of idempotents. Then $S$ is iLEF.
\end{theorem}
\begin{proof}
The statement evidently holds if $S$ is finite.

Assume that $S$ is infinite and let $H=\{h_1,\ldots,h_t\}$ be its finite subset. Consider the finite set $K = H \cup H^{-1} \cup E(S)$, where $E(S)$ is the set of idempotents of $S$. The set $K$ is symmetrised. Consider its tight pair $D_K, d_K$. By Wagner--Preston theorem $D_K$ embeds into the symmetric inverse monoid on $\Sigma$. 

Define $F_H$ to be the power semigroup of $D_K$. We will demonstrate that $F_H$ and the function $f_H: H \rightarrow F_H$ defined by $h_i \mapsto h_i d_K^{-1} M$ where $M = E(S) q_K^{-1}$ satisfy the properties of Definition \ref{def_lef}.

Evidently, $F_H$ is finite since $D_K$ is finite.

To demonstrate injectivity, assume  $h_i f_H = h_j f_H$ . Consider an arbitrary element $x_i$ of $h_i d_K^{-1}$. We have  $x_i = x_i (x_i^{ -1} x'_i) \in (h_i d_K^{-1}) M$, since $ (x_i^{-1} x_i) d_K =  h_i^{-1} h_i \in E(S)$. By our assumption we have $x_i \in (h_j d_K^{-1})M$, i.e. for some $x_j \in  h_j d_K^{-1}$ and $m \in M$ holds $x_i = x_j m$. By applying $d_K$ we have $h_i = h_j e$, where $e$ is an idempotent in $S$. Similarly, $h_j = h_i f$, where $f$ is also an idempotent in $S$. This allows us to conclude using the natural partial order on $S$ that $h_i = h_j$.

Finally, we want to show that if $h_i h_j = h_k$ for some $h_i,h_j,h_k \in H$ then $(h_i f_H)(h_j f_H) = h_k f_H$, i.e. that $(h_i d_K^{-1})M(h_j d_K^{-1})M = (h_k d_K^{-1})M$.

Consider an element of the left-handed side, which can be presented as $x_i m_1 x_j m_2$ with $m_1,m_2 \in M$, $x_i \in h_i d_K^{-1}$ and $x_j \in h_j d_K^{-1}$.

We claim that $m_1$ commutes with $x_j x_j^{ -1}$. To prove this, consider the domain $\Delta$ and range $\Lambda$ of $m_1$ and the domain/range  $\Upsilon$ of $x_j x_j^{ -1}$ as partial bijections on $\Sigma$. We claim that $\Delta=\Lambda$ and that $m_1$ sends elements of $\Delta \cap \Upsilon$ to $\Delta \cap \Upsilon$. The former follows from the fact that $m_1 m_1^{-1}$ and $m_1^{-1} m_1$ are idempotents with the same image, which means by Lemma \ref{lem_unidem} that they are the same idempotent. The latter follows from the fact that that $((x_j x_j^{-1}) m_1) q_K =  ((x_j x_j^{-1}) q_K) (m_1 q_K) \in E(S)$, allowing us to use the same argument to check that $((x_j x_j^{-1}) m_1) ((x_j x_j^{-1}) m_1))^{-1}$ (which is identity on $\Delta \cap \Upsilon$) coincides with $((x_j x_j^{-1}) m_1)^{-1} ((x_j x_j^{-1}) m_1))$ (which is identity on $(\Delta \cap \Upsilon) m_1$). It follows that $m_1$ commutes with $x_j x_j^{-1}$, with both $m_1 x_j x_j^{-1}$ and $x_j x_j^{-1} m_1$ being restrictions of $m_1$ to $\Upsilon$.
 
Thus  $x_i m_1 x_j m_2 = x_i m_1 x_j x_j^{-1} x_j m_2 = x_i x_j x_j^{-1} m_1 x_j m_2$. As $h_i h_j = h_k$, we have $x_i x_j \in h_k d_K^{-1}$. Additionally,  $x_j^{-1} m_1 x_j $ maps to $h_j^{-1} e h_j$ for some $e\in E(S)$, which is also an idempotent. Thus, the image of $x_j^{-1} m_1 x_j m_2$ is an idempotent as well, which means $x_i x_j x_j^{-1} m_1 x_j m_2 \in  (h_k d_K^{-1}) M$.

Consider an element of the right-handed side, which can be presented as $x_k m$ with $m \in M$ and $x_k \in (h_k d_K^{-1}) M$. By the multiplication property we know that for arbitrary $y_i \in  h_i d_K^{-1}$, $y_j \in  h_j d_K^{-1}$ holds $y_k:=y_i y_j \in  h_k d_K^{-1}$. We have $x_k m = x_k x_k^{-1} x_k m =$ (since both $ x_k x_k^{-1}$ and $y_k y_k^{-1} $ are idempotents with the same idempotent image) $=y_k y_k^{-1} x_k m = y_i y_j   y_k^{-1} x_k m$. It is straightforward to see that $y_k^{-1} x_k m \in M$ by the multiplication property. Thus, $x_k m = y_i y_j   y_k^{-1} x_k m = (y_i y_i^{-1} y_i ) ( y_j   y_k^{-1} x_k m) \in ((h_i d_K^{-1}) M)((h_j d_K^{-1}) M)$.

The argument above demonstrates that $S$ is LEF, however as it is inverse, it is also iLEF by Proposition \ref{prop_lefisem}.
\end{proof}

In particular, as groups are exactly inverse semigroups with a single idempotent, and a group which is an LWF group is clearly an iLWF semigroup as well, the following result, originally proven in \cite{GV}, follows.

\begin{corollary}
Let $G$ be an LWF group. Then $G$ is LEF.
\end{corollary}

\end{document}